\documentclass[journal]{IEEEtran}
\usepackage{amssymb,amsthm,bm,footnote}

\newtheorem{theorem}{Theorem}
\newtheorem{lemma}[theorem]{Lemma}

\newcommand{\overbar}[1]{\mkern 1.5mu\overline{\mkern-1.5mu#1\mkern-1.5mu}\mkern 1.5mu}

\DeclareMathAlphabet{\mathpzc}{OT1}{pzc}{m}{it}

\newcommand{\Z}{G}
\newcommand{\g}{u}
\newcommand{\X}{\mathbf X}
\newcommand{\M}{\mathbf M}

\newcommand{\Sm}{\mathbf S}

\newcommand{\bmzeta}{\bm{\zeta}}
\newcommand{\bmzetax}{\mathbf x}
\newcommand{\zetax}{x}
\newcommand{\f}{\mathbf f}
\newcommand{\z}{\mathbf z}
\newcommand{\br}{\mathbf r}
\newcommand{\ba}{\mathbf a}
\newcommand{\A}{\mathbf A}
\newcommand{\T}{\mathbf T}

\newcommand{\bv}{\mathbf v}
\newcommand{\W}{\mathbf W}
\newcommand{\bb}{\mathbf b}
\newcommand{\C}{\mathbf C}
\newcommand{\D}{\mathbf D}
\newcommand{\bLambda}{\mathbf \Lambda}
\newcommand{\bP}{\mathbf P}
\newcommand{\bzero}{\mathbf 0}
\newcommand{\J}{\mathbf J}
\newcommand{\Rsub}{{\mathbf R}_{sub}}
\newcommand{\newX}{\mathbf Y}
\newcommand{\newM}{\mathbf A}
\newcommand{\newmvec}{\mathbf a}
\newcommand{\newSm}{\mathbf B}
\newcommand{\newsvec}{\mathbf b}
\newcommand{\bvICA}{\mathbf x}
\newcommand{\buICA}{\mathbf s}
\newcommand{\AICA}{\mathbf M}
\newcommand{\baICA}{\mathbf m}
\newcommand{\pointba}{\mathbf p}

\usepackage{cite}

\ifCLASSINFOpdf
   \usepackage[pdftex]{graphicx}
   \DeclareGraphicsExtensions{.pdf,.jpeg,.png}
\else
   \usepackage[dvips]{graphicx}
   \graphicspath{{../eps/}}
   \DeclareGraphicsExtensions{.eps}
\fi

\usepackage[cmex10]{amsmath}
\begin{document}
\title{Generic uniqueness of a structured matrix factorization and applications in blind source separation}

\author{Ignat~Domanov and
        Lieven~De~Lathauwer,~\IEEEmembership{Fellow,~IEEE}
\thanks{This work was supported by Research Council KU Leuven: C1 project  c16/15/059-nD, CoE PFV/10/002 (OPTEC) and PDM postdoc grant, by F.W.O.:  project  G.0830.14N, G.0881.14N, by the Belgian Federal Science Policy Office: IUAP P7 (DYSCO II,  Dynamical systems, control
and optimization,  2012-2017), by EU: The research leading to these results has received funding from the European Research Council under the European Union's Seventh Framework Programme (FP7/2007-2013) / ERC Advanced Grant: BIOTENSORS (no.  339804). This paper reflects only the authors' views and the Union is not liable for any use that may be made of the contained information.}
\thanks{The authors are with Group Science, Engineering and Technology, KU Leuven-Kulak, E. Sabbelaan 53, 8500 Kortrijk, Belgium. Lieven~De~Lathauwer is also with Dept. of Electrical Engineering  ESAT/STADIUS KU Leuven,
 Kasteelpark Arenberg 10, bus 2446, B-3001 Leuven-Heverlee, Belgium
 (e-mail: Ignat.Domanov@kuleuven-kulak.be; Lieven.DeLathauwer@kuleuven-kulak.be).}
}


\maketitle

\begin{abstract}
Algebraic geometry, although little explored in signal processing,  provides tools that are very convenient for investigating generic properties in a wide range of applications. Generic properties are properties that hold ``almost everywhere''. We present a set of conditions that are sufficient for demonstrating the generic uniqueness of a certain structured matrix factorization. This set of conditions may be used as a checklist for generic uniqueness in different settings. We discuss two particular applications in detail. We provide a relaxed generic uniqueness condition for joint matrix diagonalization that is relevant for independent component analysis in the underdetermined case. We present generic uniqueness conditions for a recently proposed class of deterministic blind source separation methods that rely on mild source models. For the interested reader we provide some intuition on how the results are connected to their algebraic geometric roots. 
\end{abstract}


\IEEEpeerreviewmaketitle

\section{Introduction}
\subsection{Blind source separation and uniqueness}
The matrix factorization $\X=\M\Sm^T$ is well known in the blind source separation (BSS) context:
the rows of $\Sm^T$ and $\X$ represent unknown source signals and their observed linear mixtures, respectively.
The task of the BSS problem is to estimate the source matrix $\Sm$ and the  mixing matrix $\M$ from $\X$.
If no prior information is available on the matrices $\M$ or $\Sm$, then they cannot be uniquely identified  from
$\X$. Indeed, for any nonsingular matrix $\T$,
\begin{equation}\label{eq:first_eq}
\X=\M\Sm^T=(\M\T)(\Sm\T^{-T})^T=\overbar{\M}\overbar{\Sm}^T.
\end{equation}
 Applications may involve particular constraints on $\M$ and/or $\Sm$, so that in the resulting class of structured matrices the solution of \eqref{eq:first_eq} becomes unique. 
 Commonly used constraints include sparsity \cite{BDE}, constant modulus \cite{ConstModul} and Vandermonde structure  \cite{ESPIRIT1989}.

Sufficient conditions for uniqueness can be deterministic or generic. Deterministic conditions concern particular matrices $\M$ and $\Sm$. Generic conditions concern the situation that can be expected in general; a generic property is a property that holds everywhere except for a set of measure $0$. (A formal definition will be given in Subsection \ref{subsectio_ID} below.)

To illustrate the meaning of deterministic and generic uniqueness let us consider  decomposition \eqref{eq:first_eq} in which
$\X\in \mathbb{C}^{K\times N}$, $\M\in \mathbb{C}^{K\times R}$ and the columns of  $\Sm\in\mathbb C^{N\times R}$
are obtained by sampling the exponential signals $z_1^{t-1},\dots,z_R^{t-1}$ at $t=1,\dots,N$. Then $(\Sm)_{nr}=(z_r^{n-1})$, i.e. $\Sm$ is a Vandermonde matrix.
A deterministic condition under which decomposition \eqref{eq:first_eq} is unique (up to trivial indeterminacies) is \cite{ESPIRIT1989}: (i) the Vandermonde matrix $\Sm$ has strictly more rows than columns and its generators $z_j$ are distinct and (ii) the matrix $\M$ has full column rank. (In this paper we say that 
an $K\times R$ matrix {\em has full column rank} if its column rank is $R$, which implies
 $K\geq R$.) This deterministic condition can easily be verified for any particular $\M$ and $\Sm$. A generic variant  is: (i) the Vandermonde matrix $\Sm$ has $N > R$ and (ii) the (unstructured) matrix $\M$ has $K \geq R$. Indeed, under these dimensionality conditions the deterministic conditions are satisfied everywhere, except in a set of measure $0$ (which contains the particular cases of coinciding generators $z_r$ and the cases in which the columns of $\M$ are not linearly independent despite the fact that $\M$ is square or even tall).
 Note that generic properties do not allow one to make statements about specific matrices; they only show the general picture.

As mentioned before, BSS has many variants, which differ in the types of constraints that are imposed. Different constraints usually mean different deterministic uniqueness  conditions, and the derivation of these is work that is difficult to automate. In this paper we focus on generic uniqueness conditions. We propose a framework with which generic uniqueness can be investigated in a broad range of cases. Indeed, it will become clear that if we limit ourselves to generic uniqueness,
the derivation of conditions can to some extent be automated. We discuss two concrete applications which may serve as examples.

Our approach builds on results in algebraic geometry. Algebraic geometry has so far been used in system theory in \cite{Hermann1977,Hermann1977PartII} and it also has direct applications in tensor-based BSS via the generic uniqueness of tensor decompositions \cite{LievenCichocki2013,Comon2014,AlgGeom1}. Our paper makes a contribution in further connecting algebraic geometry with applications in signal processing.
\subsection{Notation} 
Throughout the paper $\mathbb F$ denotes the field of real or complex numbers;
bold lowercase letters denote vectors, while bold uppercase letters represent matrices; 
a column of a matrix $\A$ and an entry of a vector $\bb$ are denoted by $\ba_j$ 
and $b_j$, respectively; the superscripts $\cdot^*$, $\cdot^T$ and  $\cdot^H$ are used for the conjugate, transpose, and Hermitian transpose, respectively; ``$\otimes$'' denotes the Kronecker product.

\subsection{Statement of the problem and organization of the paper}\label{subsectionIC}
{\em A structured matrix factorization.} 
In this paper we consider the following structured factorization  of a $K\times N$ matrix $\newX$, 
\begin{equation}\label{eq:sumstrucrank1}
\newX =
\newM(\z)\newSm(\z)^T,
\quad \z\in\Omega
\end{equation}
where $\Omega$ is a subset of $\mathbb F^n$ and $\newM(\z)$ and $\newSm(\z)$  are known matrix-valued functions defined on $\Omega$.

 W.l.o.g. we can assume that the parameter vector $\z=[z_1\ \dots\ z_n]^T$ is ordered such that $\newSm(\z)$ depends on the last $s\leq n$ entries, while
 $\newM(\z)$ depends on $m\geq 0$ entries that are not necessarily the first or the last. That is,  
\begin{equation*} 
\newM(\z)=\newM(z_{i_1},\dots,z_{i_m}),\ \ 
\newSm(\z)=\newSm(z_{n-s+1},\dots,z_n)
\end{equation*} 
for some $1\leq i_1<i_2<\dots<i_m\leq n$.
In general, the entries used to parameterize $\newM$ and $\newSm$ are allowed to overlap so that $m+s\geq n$.
The case where $\newM$ and $\newSm$ depend on separated parameter sets corresponds to  $m+s=n$; in this case
$\newM$ depends strictly on the first $m$ of the entries of $\z$. 

Our study is limited to $K\times R$ matrices $\newM(\z)$ that generically have full column rank. We do not make any other assumptions on the form of $\newM(\z)$. In particular, we do not impose restrictions on how the entries depend on $\z$. We are however more explicit about the form of the $N\times R$ matrix $\newSm(\z)$.
We assume that each of its columns $\newsvec_r(\z)$ is generated by $l$ parameters that are
independent of the parameters used to generate the other columns, i.e., $\newSm(\z)=[\newsvec_1(\bmzeta_1)\ \dots\ \newsvec_R(\bmzeta_R)]$ with $\bmzeta_1,\dots,\bmzeta_R\in\mathbb F^l$. Note that the independence implies that $s=Rl$ and that 
$[\bmzeta_1^T\ \dots\ \bmzeta_R^T]^T$  and $[z_{n-s+1}\ \dots\ z_n]^T$ are the same up to index permutation.

For the sake of exposition, let us first  consider a class of matrices $\newSm(\z)$ that is smaller  than the class that we will be able to handle in our derivation of generic uniqueness conditions. Namely, let us first consider matrices
$\newSm^{rat}(\z)$, of which the $n$-th row is obtained by evaluating a known rational function $\frac{p_n(\cdot)}{q_n(\cdot)}$
at some points $\bmzeta_1,\dots,\bmzeta_R \in \mathbb F^l$, $1 \leq n \leq N$:
\begin{equation*}
 \newSm^{rat}(\z)
  =
 \left[\begin{matrix}
 \frac{p_1(\bmzeta_1)}{q_1(\bmzeta_1)}&\dots & \frac{p_1(\bmzeta_R)}{q_1(\bmzeta_R)}\\
 \vdots&\vdots&\vdots\\
 \frac{p_N(\bmzeta_1)}{q_N(\bmzeta_1)}&\dots & \frac{p_N(\bmzeta_R)}{q_N(\bmzeta_R)}
 \end{matrix}\right], 
  \end{equation*}
  where
  $$
   p_1,\dots,p_N,\ q_1,\dots,q_N\ \text{ are polynomials in  }l \text{ variables.}
   $$
 Note that we model  a column of $\newSm^{rat}$ through the values taken by $N$ functions $\frac{p_1(\cdot)}{q_1(\cdot)},\dots,\frac{p_N(\cdot)}{q_N(\cdot)}$ at one particular point $\bmzeta_r$. On the other hand, a row of $\newSm^{rat}$ is modeled as values taken by one particular function $\frac{p_n(\cdot)}{q_n(\cdot)}$
    at $R$ points $\bmzeta_1,\dots,\bmzeta_R$.
    
 The structure that we consider in our study for the $N\times R$ matrix $\newSm(\z)$ is more general than the rational structure of $\newSm^{rat}(\z)$ in the sense that we additionally allow (possibly nonlinear) transformations of $\bmzeta_1,\dots,\bmzeta_R$. Formally, we assume that the columns of $\newSm(\z)$ are sampled values of known vector functions of the form
\begin{equation}\label{eq:column_s}
 \newsvec(\bmzeta) = \left[\frac{p_1(\f(\bmzeta))}{q_1(\f(\bmzeta))}\ \dots\ \frac{p_N(\f(\bmzeta))}{q_N(\f(\bmzeta))}\right]^T,\quad
  \bmzeta\in\mathbb F^l,
 \end{equation}
at  points $\bmzeta_1,\dots,\bmzeta_R\in\mathbb F^l$, such that
\begin{align*}
 \newSm(\z)&=[\newsvec(\bmzeta_1) \dots \newsvec(\bmzeta_R)]=
 \left[\begin{matrix}
 \frac{p_1(\f(\bmzeta_1))}{q_1(\f(\bmzeta_1))}&\dots & \frac{p_1(\f(\bmzeta_R))}{q_1(\f(\bmzeta_R))}\\
 \vdots&\vdots&\vdots\\
 \frac{p_N(\f(\bmzeta_1))}{q_N(\f(\bmzeta_1))}&\dots & \frac{p_N(\f(\bmzeta_R))}{q_N(\f(\bmzeta_R))}
 \end{matrix}\right], 
  \end{align*}
	 where
     \begin{equation*} 
      \begin{split}
      &\f(\bmzeta)=(f_1(\bmzeta),\dots,f_l(\bmzeta)) \in\mathbb F^l,\\
      & f_1,\dots,f_l \text{ are scalar functions of }l\text{ variables.}
      \end{split}
      \end{equation*}
  The functions $f_1,\dots,f_l$ are subject to an analyticity assumption that will be specified in Theorem \ref{th:main} further. Although our general result in Theorem \ref{th:main} will be formulated in terms of functions $f_1,\dots,f_l$  in $l$ variables, in the applications in Sections \ref{section_ICA}--\ref{sampled sources} we will only need entry-wise transformations:
     \begin{equation}\label{eq:single_var}
     \f({\bmzeta})=\f(\zeta_1,\dots,\zeta_l)=(f_1(\zeta_1),\dots,f_l(\zeta_l))
     \end{equation}
with $f_1,\dots,f_l$ analytic functions in one variable.

  As an example of how the model for $\newSm(\z)$ can be used,
   consider $R$ 
  vectors that  are obtained 
    by sampling the  exponential signals $e^{i \zeta_1 (t-1)},\dots, e^{i \zeta_R (t-1)}$ (with $\zeta_1,\dots,\zeta_R\in\mathbb R$) at $t=1,\dots,N$. In this case 
     $\newSm(\z)$ is an $N\times R$ Vandermonde matrix with unit norm generators; its  $r$th column is $\newsvec(\zeta_r)=[1\ e^{i \zeta_r}\ \dots\ e^{i \zeta_r (N-1)}]^T$.
        We have $e^{i \zeta_r (n-1)}=\frac{p_n(f(\zeta_r))}{q_n(f(\zeta_r))}$, where
      $f(\zeta)=e^{i \zeta}$, $p_n(\zetax)=\zetax^{n-1}$, and $q_n(\zetax)=1$ for $\zeta\in \mathbb R$
       and $\zetax\in \mathbb C$. 
      
{\em Generic uniqueness of the decomposition.}\label{subsectio_ID}
We interpret factorization \eqref{eq:sumstrucrank1}  as  a decomposition into a sum of structured rank-$1$ matrices
 \begin{equation}\label{eq:sumstrucrank1_structured}
 \begin{split}
 \newX =\newM(\z)\newSm(\z)^T= 
 \sum\limits_{r=1}^R\newmvec_r(\z)\newsvec(\bmzeta_r)^T,\ \z\in\Omega, 
 \end{split}
 \end{equation}
where $\newmvec_r(\z)$ denotes the $r$th column of $\newM(\z)$. 
It is clear that in \eqref{eq:sumstrucrank1_structured} the rank-1 terms can be arbitrarily permuted.
We say that {\em decomposition \eqref{eq:sumstrucrank1_structured} is unique} when it is only subject to this trivial indeterminacy.
We say that  decomposition \eqref{eq:sumstrucrank1_structured} is  {\em generically unique} if
it is unique for a generic choice of $\z\in\Omega$, that is 
\begin{equation}\label{eq:def_gen_uniq}
\mu_n\{\z\in\Omega:\ \text{decomposition \eqref{eq:sumstrucrank1_structured} is not unique}\}=0,
\end{equation}
where $\mu_n$ is a measure that is absolutely continuous (a.c.) with respect to the Lebesgue measure on $\mathbb F^n$.

In this paper we  present conditions on the polynomials $p_1,\dots,p_N$, $q_1,\dots,q_N$, the function $\f$  and the set $\Omega$ which guarantee that decomposition \eqref{eq:sumstrucrank1_structured} is generically unique. As a technical assumption, since in the case where $\mu_n(\Omega)=0$ condition \eqref{eq:def_gen_uniq} cannot be used to infer generic uniqueness from a subset of $\Omega$,  we assume that $\mu_n(\Omega)>0$.

{\em Organization and results.} In Section \ref{subsection_2_B} we state the main result of this paper in general terms, namely, Theorem \ref{th:main} presents conditions that guarantee that the structured decomposition \eqref{eq:sumstrucrank1_structured} is generically unique. The proof of Theorem \ref{th:main} is given in Appendix \ref{Appendix_A}. Besides the technical derivation, Appendix \ref{Appendix_A} provides some intuition behind the high-level reasoning and makes the connection with the trisecant lemma in algebraic geometry, for readers who are interested. In Sections \ref{section_ICA}--\ref{sampled sources} we use Theorem \ref{th:main} to obtain new uniqueness results in the context of two different applications. This is done by first expressing the specific BSS problem as a decomposition of the form \eqref{eq:sumstrucrank1_structured}, for which the list of conditions in Theorem \ref{th:main} is checked. Section \ref{section_ICA} concerns an application in independent component analysis. More precisely, it concerns joint matrix diagonalization in the underdetermined case (more sources than observations) and presents a new, relaxed bound on the number of sources under which the solution of this basic subproblem is generically unique. This bound is a simple expression in the number of matrices and their dimension. 
 Section \ref{sampled sources} presents generic uniqueness results for a recently introduced class of deterministic blind source separation algorithms that may be seen as a variant of sparse component analysis which makes use of a non-discrete dictionary of basis functions. Appendix \ref{Appendix_B} contains the short proof of a technical lemma in Section \ref{sampled sources}. The paper is concluded in Section \ref{sec:conclusion}.
\section{Main result}\label{subsection_2_B}
The following theorem is our main result on  generic uniqueness of decomposition \eqref{eq:sumstrucrank1_structured}.
It states that, generically, 
 the $R$ structured rank-$1$ terms of the $K\times N$ matrix $\newX$ can be uniquely recovered if $K\geq R$ and
  $R\leq\widehat{N}-\widehat{l}$. Here,
$\widehat{N}\leq N$ is a lower bound on the dimension of the linear vector space  
$\operatorname{span}\{\br(\bmzetax):\ q_1(\bmzetax)\cdots q_N(\bmzetax)\ne 0,\ \bmzetax\in\mathbb F^l\}$
generated by vectors of the form
\begin{equation}\label{eq:column_r}
 \br(\bmzetax) = \left[\frac{p_1(\bmzetax)}{q_1(\bmzetax)}\ \dots\ \frac{p_N(\bmzetax)}{q_N(\bmzetax)}\right]^T.
 \end{equation}
 (Note that the definition of $\br(\bmzetax)$ does not involve a nonlinear transformation $\f$, even when such a nonlinear transformation is used for modelling $\newsvec(\bmzeta)$.) On the other hand, 
  the value $\widehat{l}\leq l$ is an upper bound on the number of ``free parameters'' actually needed to parameterize  a generic vector of the form \eqref{eq:column_r}. (Indeed, although $\br(\bmzetax)$ is generated by $l$ independent parameters, it may be possible to do it with less in particular cases. For instance, let $N=3$, 
  $q_1(\bmzetax)=q_2(\bmzetax)=q_3(\bmzetax)=1$ and $p_1(\bmzetax)=\zetax_1+\zetax_3$, 
  $p_2(\bmzetax)=\zetax_2-\zetax_3$, $p_3(\bmzetax)=\zetax_1+\zetax_2$, so that
  $\br(\bmzetax)=\W\bmzetax$ with 
                                 $\W=\begin{scriptsize}
                                 \left[\begin{array}{rrr}
                                                                  1 & 0 & 1\\
                                                                  0 & 1 & -1\\
                                                                  1 & 1 & 0\end{array}\right]
                                 \end{scriptsize}
                                 $.
  Since $\operatorname{rank}(\W)=2$, $\br(\bmzetax)$ can be parameterized by $2<3$ independent parameters.)
  
  In the theorem and throughout the paper we use  $\J(\br,\bmzetax)\in\mathbb F^{N\times l}$ and 
  $\J(\f,\bmzeta)\in\mathbb F^{l\times l}$ to denote the Jacobian matrices of $\br$ and $\f$, respectively,
  $$
  \left(\J(\br,\bmzetax)\right)_{ij}=\frac{\partial\frac{p_i}{q_i}}{\partial\zetax_j},\quad
  \left(\J(\f,\bmzeta)\right)_{ij}=\frac{\partial f_i}{\partial\zeta_j}.
  $$
  Further, 
  $$
  \operatorname{Range}(\br)=\{\br(\bmzetax):\ q_1(\bmzetax)\cdots q_N(\bmzetax)\ne 0,\ \bmzetax\in\mathbb C^l\}\subset\mathbb C^N
  $$
  denotes the set of all values of $\br(\bmzetax)$ for $\bmzetax\in\mathbb C^l$. We say that the set $\operatorname{Range}(\br)$
  is invariant under scaling if
  $$
  \operatorname{Range}(\br)\supseteq \lambda\cdot\operatorname{Range}(\br)\ \text{for all}\ \lambda\in\mathbb C.
  $$ 
  \begin{theorem} 
\label{th:main}
 Let $\Omega$ be a subset of $\mathbb F^n$ and $\mu_n(\Omega)>0$.
Assume that
\begin{enumerate}[\IEEEsetlabelwidth{Z}]
\item the matrix $\newM(\z)$ has full column rank
 for  a generic choice of $\z\in\Omega$, that is,
 \begin{equation*}
 \mu_n\{\z\in\Omega:\ \operatorname{rank}\newM(\z)<R\}=0;
 \end{equation*}
\item 
the coordinate functions $f_1,\dots,f_l$ of $\f$  can be represented as
$$
f_1(\bmzeta)= \frac{f_{1,num}(\bmzeta)}{ f_{1,den}(\bmzeta)},\dots,
f_l(\bmzeta)= \frac{ f_{l,num}(\bmzeta)}{ f_{l,den}(\bmzeta)},
$$
 where the functions
\begin{equation*} 
f_{1,num}(\bmzeta),f_{1,den}(\bmzeta),\dots,f_{l,num}(\bmzeta),f_{l,den}(\bmzeta)
\end{equation*}
are analytic on $\mathbb C^l$;  
\item there exists $\bmzeta^0\in \mathbb C^l$ such that $\det\J(\f,\bmzeta^0)\ne 0$; 
\item the dimension of the subspace spanned by the vectors  of  form \eqref{eq:column_r} is at least 
$\widehat{N}$,
$$
\dim\operatorname{span}\{\br(\bmzetax):\ q_1(\bmzetax)\cdots q_N(\bmzetax)\ne 0,\ \bmzetax\in\mathbb C^l\}\geq \widehat{N};
$$
\item $\operatorname{rank}\J(\br,\bmzetax)\leq\widehat{l}$ for a generic choice of $\bmzetax\in\mathbb C^l$;
\item $R\leq\widehat{N}-\widehat{l}$ or $R\leq\widehat{N}-\widehat{l}-1$,
depending on whether the set $\operatorname{Range}(\br)$ is invariant under scaling or not.
 \end{enumerate}
Then decomposition \eqref{eq:sumstrucrank1_structured} is generically unique.
\end{theorem}
\begin{proof}
See Appendix A.
\end{proof}
Assumptions 1--6 can be used as a checklist for demonstrating the generic uniqueness of decompositions that can be put in the form \eqref{eq:sumstrucrank1}. We will discuss two application examples in Sections \ref{section_ICA}--\ref{sampled sources}.
We comment on  the following  aspects of assumptions 2--6. 

$\bullet$ 
 In this paper we will use Theorem \ref{th:main} in the case
where  $\f(\bmzeta)$ is of the form \eqref{eq:single_var}.
For such $\f$ the matrix $\J(\f,\bmzeta)$ is  diagonal, yielding that
$\det\J(\f,\bmzeta)=f_1'(\zeta_1)\cdots f_l'(\zeta_l)$. Moreover, in this paper $f_1,\dots,f_l$ are non-constant,
so $\det\J(\f,\bmzeta)$ is not identically zero. 
Thus, {\em assumption} 3 in Theorem \ref{th:main} will hold automatically.

$\bullet$ 
For the reader who wishes to apply Theorem \ref{th:main} in  cases where $\f$ is not of the form \eqref{eq:single_var}, we  recall the definition of an analytic (or holomorphic) function of several variables  used in {\em assumption 2}. 
 A function  $f:\mathbb C^l\rightarrow \mathbb C$ of $l$ complex variables is analytic \cite[page 4]{Krantz2001} if it is analytic in each variable separately, that is, 
if for each $j=1,\dots,l$ and accordingly fixed $\zeta_1,\dotsm\zeta_{j-1},\zeta_{j+1},\dots,\zeta_l$ the function
$$
z\mapsto f(\zeta_1,\dotsm\zeta_{j-1},z,\zeta_{j+1},\dots,\zeta_l)
$$
is analytic on $\mathbb C$ in the classical one-variable sense. Examples of analytic functions
of several variables can be obtained by taking compositions of multivariate polynomials and analytic functions in one variables,
e.g. $f(\zeta_1,\zeta_2)=\sin(\cos (\zeta_1\zeta_2))+\zeta_1$.

$\bullet$ To check {\em assumption 4} in  Theorem \ref{th:main} it is sufficient to present (or prove the existence of) $\widehat{N}$ linearly independent vectors $\{\br(\bmzetax_i)\}_{i=1}^{\widehat{N}}$. It is clear that larger $\widehat{N}$ yield a better bound on $R$ in {\em assumption} 6. 
In all cases considered in this paper $\widehat{N}=N$. The situation  $\widehat{N}<N$ may appear when
the $N\times 1$ vector-function $\newsvec(\bmzeta)$ models a periodic, (locally) odd or even function, etc. 

$\bullet$ The goal of {\em assumption 5} is to check whether  generic signals of the form \eqref{eq:column_r}
can be re-parameterized with  fewer (i.e. $\widehat{l}<l$) parameters. In this case, the Jacobian $\J(\br,\bmzetax)$ has indeed rank strictly less than $l$. It is clear that {\em assumption} 5 in Theorem \ref{th:main} holds trivially for $\widehat{l}=l$ and that smaller $\widehat{l}$ yield a better bound on $R$ in {\em assumption} 6.
In this paper we  set either $\widehat{l}=l$ (namely in the proof of Theorem \ref {th:exp_poly}) or, in the case where it is clear that  $\J(\br,\bmzetax)$  does not have full column rank (namely in the proof of Theorems \ref{th:ICA} and \ref{th:rat_funct}), $\widehat{l}=l-1$.

$\bullet$ 
Although the Theorem holds both for 
$\mathbb F=\mathbb C$ and $\mathbb F=\mathbb R$, we formulated {\em assumptions} 3, 4 and 5 in Theorem \ref{th:main} for $\bmzeta^0\in\mathbb C^l$ and $\bmzetax\in\mathbb C^l$.
In these assumptions $\mathbb C^l$ can also be replaced by
$\mathbb R^l$. We presented the complex variants, even for the case $\mathbb F=\mathbb R$, 
since they may be easier to verify than their real counterparts, as $\bmzeta^0$ and $\bmzetax$ are allowed to take values in a larger set. On the other hand, the analyticity on $\mathbb C^l$
 in {\em assumption} 2 is a stronger assumption than analyticity on $\mathbb R^l$ and is needed in the form it is given.

\section{An application in independent component analysis}\label{section_ICA}
We consider data described by the  model
$
\bvICA=\AICA\buICA
$,
where $\bvICA$ is the $I$-dimensional vector of observations, $\buICA$ is the $R$-dimensional unknown source vector and $\AICA$ is the $I$-by-$R$ unknown
mixing matrix. We assume that the sources are mutually uncorrelated  but individually correlated in time. 
 It is known that the spatial covariance matrices of the observations  satisfy \cite{1997SOBI}
\begin{IEEEeqnarray}{rCl}
\C_1&= \operatorname{E}(\bvICA_t\bvICA^H_{t+\tau_1})=\AICA\D_1\AICA^H=\sum\limits_{r=1}^Rd_{1r}\baICA_r\baICA_r^H,\IEEEnonumber\\
&  \vdots
\label{eq:SOBI}\\
\C_P&= \operatorname{E}(\bvICA_t\bvICA^H_{t+\tau_P})=\AICA\D_P\AICA^H=\sum\limits_{r=1}^Rd_{Pr}\baICA_r\baICA_r^H,\IEEEnonumber
\end{IEEEeqnarray}
in which $\D_p=\operatorname{E}(\buICA_t\buICA_{t+\tau_p}^H)$ is the $R$-by-$R$ diagonal matrix
with the elements of the vector $(d_{p1},\dots,d_{pR})$ on the main diagonal.
The estimation of $\AICA$ from the set $\{\C_p\}$ is known as
 Second-Order Blind Identification (SOBI) \cite{1997SOBI} or as Second-Order Blind Identification of Underdetermined Mixtures 
 (SOBIUM)\cite{2008LievenSOBIUM} depending on whether the matrix $\AICA$ has full column rank or not. Variants of this problem are discussed in, e.g., \cite{Pham_Cardoso_2001},\cite{Yeredor2002},\cite{Yeredor2000},\cite[Chapter 7]{ComoJ10}. 
It is clear that if the matrices $\AICA$ and $\D_1,\dots,\D_P$ satisfy \eqref{eq:SOBI}, then the matrices
$\overline{\AICA}=\AICA{\bLambda}\bP$ and $\overline{\D}_1=\bP^T\D_1\bP,\dots,\overline{\D}_P=\bP^T\D_P\bP$ also  satisfy \eqref{eq:SOBI} for any permutation matrix $\bP$ and diagonal unitary matrix
${\bLambda}$. 
We say that \eqref{eq:SOBI} has a unique solution  when it is only subject to this trivial indeterminacy.

Generic uniqueness of solutions of \eqref{eq:SOBI} has been studied 1) in \cite{Psycho2006} and \cite[Subsection 1.4.2]{AlgGeom1}
in the case where the superscript ``$H$'' in \eqref{eq:SOBI} is replaced by the superscript ``$T$'' 
(for quantities $\bvICA$, $\AICA$ are $\buICA$ that can be  either real valued or complex valued); 2) in
 \cite{2008LievenSOBIUM}, \cite{MikaelnewUniq3rdorder} (where $\bvICA$, $\AICA$ are $\buICA$ are complex valued).
In \cite{AlgGeom1,2008LievenSOBIUM,MikaelnewUniq3rdorder} the matrix equations in \eqref{eq:SOBI} were interpreted as 
 a so-called canonical polyadic decomposition of a (partially symmetric) tensor.  
In the following theorems we interpret  the equations in \eqref{eq:SOBI} as matrix factorization problem \eqref{eq:sumstrucrank1}. The new interpretation only relies on elementary linear algebra; it does not make use of advanced  results on tensor decompositions while it does lead to more relaxed bounds on $R$ than in   \cite{2008LievenSOBIUM},\cite{MikaelnewUniq3rdorder} for $I\geq 5$. We consider  the variants  $\tau_p\ne 0$, $1\leq p\leq P$, and $\tau_1=0$ in Theorems \ref{th:ICA} and \ref{th:ICAtau_1ne0}, respectively.
\begin{theorem}\label{th:ICA}
Assume that  $\tau_1\ne 0$ and        
\begin{equation}\label{eq:sobium_bound}
R\leq \min(2P, (I-1)^2).
\end{equation}
Then \eqref{eq:SOBI} has a unique solution for generic matrices
$\AICA$ and $\D_1,\dots,\D_P$, i.e., 
\begin{equation}\label{eq:ICA_measure}
\mu_k\{ (\operatorname{vec}(\D),\operatorname{vec}(\AICA)):\ \text{solution of } \eqref{eq:SOBI} \text{ is not unique}\}=0,
\end{equation}
where $\D$ denotes the $P\times R$ matrix with entries $d_{pr}$,  $k=IR+PR$, and $\mu_k$
is a measure that is a.c. with respect to the Lebesgue measure on $\mathbb C^k$ .
\end{theorem}
\begin{proof} \textup{(i)}\ 
 First we rewrite the equations in \eqref{eq:SOBI} as  matrix decomposition \eqref{eq:sumstrucrank1_structured}\footnotemark[1].
 \footnotetext[1]{Our derivation of a matrix version of \eqref{eq:SOBI} is similar to the derivation in \cite[Subsection 5.2]{MikaelnewUniq3rdorder}.}
In step \textup{(ii)} we will apply Theorem \ref{th:main} to  \eqref{eq:sumstrucrank1_structured}.

   Since  $\C_p^H=\sum\limits_{r=1}^Rd_{pr}^*\baICA_r\baICA_r^H$,  the $p$th equation in \eqref{eq:SOBI} is equivalent to the following pair of equations
\begin{gather*}
\operatorname{Re}\C_p=\frac{\C_p+\C_p^H}{2} =
\sum_{r=1}^R\operatorname{Re}d_{pr}\baICA_r\baICA_r^H,\\
\operatorname{Im}\C_p=\frac{\C_p-\C_p^H}{2i} = 
\sum_{r=1}^R\operatorname{Im}d_{pr}\baICA_r\baICA_r^H.
\end{gather*}
Since $\operatorname{vec}(\baICA\baICA^H)=\baICA^*\otimes\baICA$, we further obtain that
\begin{align*}
\operatorname{vec}(\operatorname{Re}\C_p)^T& =\\
& [\operatorname{Re}d_{p1} \dots \operatorname{Re}d_{pR}] [\baICA_1^*\otimes \baICA_1 \dots \baICA_R^*\otimes \baICA_R]^T,\\
\operatorname{vec}(\operatorname{Im}\C_p)^T& =\\
& [\operatorname{Im}d_{p1} \dots \operatorname{Im}d_{pR}] [\baICA_1^*\otimes \baICA_1 \dots \baICA_R^*\otimes \baICA_R]^T.
\end{align*}
Hence, the $P$ equations in \eqref{eq:SOBI} can be rewritten  as $\newX=\newM\newSm^T$, where 
\begin{align*}
&\newX=\\
&[ \operatorname{vec}(\operatorname{Re}\C_1) \dots \operatorname{vec}(\operatorname{Re}\C_P)\ 
\operatorname{vec}(\operatorname{Im}\C_1) \dots
 \operatorname{vec}(\operatorname{Im}\C_P)]^T,\\
&\newM=
 \left[\begin{matrix}
 \frac{\D+\D^*}{2}\\
 \frac{\D-\D^*}{2i}
 \end{matrix}\right]\in\mathbb R^{K\times R},\ K=2P,\  \text{and }\\
&\newSm=[\baICA_1^*\otimes \baICA_1\ \dots\ \baICA_R^*\otimes \baICA_R]\in\mathbb R^{N\times R},\ N=I^2. 
 \end{align*}
 Now we choose $l$, $\bmzeta$, $p_n$, $q_n$, and $\f$ such that  the columns of $\newSm$ are of  the form  \eqref{eq:column_s}.
 Note that the trivial parameterization $\newsvec(\bmzeta)= \bmzeta^*\otimes\bmzeta$ with $\bmzeta\in\mathbb C^I$ is not of the form \eqref{eq:column_s} because of the conjugation. However, 
since for $\baICA=\operatorname{Re}\baICA+i\operatorname{Im}\baICA$,
 $$
 \baICA^*\otimes \baICA=(\operatorname{Re}\baICA-i\operatorname{Im}\baICA)\otimes
 (\operatorname{Re}\baICA+i\operatorname{Im}\baICA),
 $$
 the parameterization 
 \begin{equation*}
 \begin{split}
  \newsvec(\bmzeta)=&([\zeta_1\ \dots\ \zeta_I]^T-i[\zeta_{i+1}\ \dots\ \zeta_{2I}]^T\otimes\\
&([\zeta_1\ \dots\ \zeta_I]^T+i[\zeta_{i+1}\ \dots\ \zeta_{2I}]^T),\quad  \bmzeta\in\mathbb R^l
\end{split}
  \end{equation*}
  with   $l=2I$,
  is of the  form \eqref{eq:column_s}. As a matter of fact, each component of $\newsvec(\bmzeta)$
  is a polynomial $p_n$ in $\zeta_1,\dots,\zeta_l$, $1\leq n\leq N$,  so we can set  $\f(\bmzeta)=\bmzeta$, and
  $q_1(\bmzeta)=\dots=q_{N}(\bmzeta)=1$.
 
It is clear that the  matrix $\newM$ can be parameterized independently of $\newSm$ by $m=2PR$ real parameters, namely, by the entries of the $P\times R$ matrices $\frac{\D+\D^*}{2}$ and $\frac{\D-\D^*}{2i}$. Thus,  the equations in \eqref{eq:SOBI} can be rewritten  as decomposition \eqref{eq:sumstrucrank1_structured} with $\z\in\Omega=\mathbb R^n$, where 
  $
  n=m+s=2PR+lR=2PR+2IR
  $.
Moreover, one can easily verify that \eqref{eq:SOBI} has a unique solution if and only if 
decomposition \eqref{eq:sumstrucrank1_structured} is unique.
In turn, since, obviously, \eqref{eq:ICA_measure} is equivalent to
\begin{equation*}
\begin{split}
&\mu_n\left\{\left(\operatorname{vec}((\D+\D^*)/2),\operatorname{vec}((\D-\D^*)/2i),\right.\right.\\
&\qquad\qquad\qquad\left.\operatorname{Re}\baICA_1,\operatorname{Im}\baICA_1,\dots,
\operatorname{Re}\baICA_R,\operatorname{Im}\baICA_R\right):\\
&\qquad\qquad\qquad\qquad\left. \text{solution of } \eqref{eq:SOBI} \text{ is not unique}\right\}=0,
\end{split}
\end{equation*}
it follows that \eqref{eq:ICA_measure} can be rewritten as \eqref{eq:def_gen_uniq}.

\textup{(ii)} To prove \eqref{eq:def_gen_uniq}  we check assumptions 1--6 in Theorem \ref{th:main}.
Assumption 1: it is clear that if $\D$ is generic, then, by the assumption $2P\geq R$, the matrix $\newM$ has full column rank. Assumptions 2--3 are trivial   since $\f$ is the identity mapping. Assumption 4:
  since the rank-$1$ matrices  of the form $\baICA\baICA^H$ span the  whole space of $I\times I$ matrices and $\newsvec(\operatorname{Re}\baICA,\operatorname{Im}\baICA)=\operatorname{vec}(\baICA\baICA^H)$ it follows that assumption 4 holds for $\widehat{N}=I^2$.  
    Assumption 5: an elementary computation shows that for a generic $\bmzeta$, 
  $\J(\br,\bmzetax)[\zetax_{I+1}\ \dots\ \zetax_{2I}\ -\zetax_1\ \dots\ -\zetax_I]=\bf 0$, implying that $\operatorname{rank}{(\J(\br,\bmzetax))}\leq l-1$, so we set  $\widehat{l}=l-1$.
  Assumption 6:
     since $\widehat{N}-\widehat{l} =I^2-2I+1$, assumption 6 holds by \eqref{eq:sobium_bound} since
   $\lambda\br(\bmzeta)=\lambda\newsvec(\bmzeta)=\newsvec(\sqrt{\lambda}\bmzeta)
  =\br(\sqrt{\lambda}\bmzeta)$.
  \end{proof}
  Now we consider the case $\tau_1=0$. The only difference with the case $\tau_1\ne 0$ is that the  diagonal matrix $\D_1=\operatorname{E}(\buICA_t\buICA_{t+\tau_1}^H)$  is  real, yielding that 
    \eqref{eq:SOBI}
  can be parameterized by $R$ real and $IR+(P-1)R$ complex parameters, or equivalently, by  
$n=R+2IR+2(P-1)R$ real parameters.
   \begin{theorem}\label{th:ICAtau_1ne0}
   Assume that  $\tau_1=0$ and        
$ R\leq \min(2P-1, (I-1)^2).$
 Then \eqref{eq:SOBI} has a unique solution for generic  real matrix $\D_1$ and generic complex matrices
 $\AICA$ and $\D_2,\dots,\D_P$, i.e.,
 \begin{equation*}
 \begin{split}
 &\mu_n\left\{d_{11},\dots,d_{1R},\left(\operatorname{vec}((\overline{\D}+\overline{\D}^*)/2),\operatorname{vec}((\overline{\D}-\overline{\D}^*)/2i),\right.\right.\\
 &\qquad\qquad\qquad\left.\operatorname{Re}\baICA_1,\operatorname{Im}\baICA_1,\dots,
 \operatorname{Re}\baICA_R,\operatorname{Im}\baICA_R\right):\\
 &\qquad\qquad\qquad\qquad\left. \text{solution of } \eqref{eq:SOBI} \text{ is not unique}\right\}=0,
 \end{split}
 \end{equation*}
where $\overline{\D}\in\mathbb C^{(P-1)\times R}$ denotes a matrix with entries $d_{pr}$ ($p>1$), 
$n=(2I+2P-1)R$, and $\mu_n$
is a measure that is a.c. with respect to the Lebesgue measure on $\mathbb R^n$ . 
 \end{theorem}
 \begin{proof}
 The proof is essentially the same as that of Theorem \ref{th:ICA}.
  \end{proof}
  Assuming that $R\leq P$, we check up to which value of $R$ condition \eqref{eq:sobium_bound} in Theorem \ref{th:ICA}
 and conditions $R(R-1)\leq I^2(I-1)^2/2$ in \cite{2008LievenSOBIUM} and  $R\leq (I^2-I)/2$ in \cite{AlgGeom1} hold. The results  are shown
 in Table \ref{table:ICA}. Note that under the condition in \cite{2008LievenSOBIUM} the mixing matrix $\AICA$ can be found from an eigenvalue decomposition in the exact case. Hence, it is not surprising that this condition is more restrictive. The condition in \cite{AlgGeom1} is more restrictive since, if $\D_p$ is complex, the unsymmetric matrix $\M\D_p\M^H$ has more distinct entries than the complex symmetric matrix $\M\D_p\M^T$.
  \begin{table}[!t]
 \renewcommand{\arraystretch}{1.3}
  \centering
         \caption{Upper bounds on the number  of sources in SOBI}
         \label{table:ICA}
         \begin{tabular}{|l|c||c|c|c|c|c|c|c|}
             \cline{2-9}
              \multicolumn{1}{c|}{} & \multicolumn{1}{|c||}{$I$} & 3 & 4 & 5 & 6 & 7 & 8 & 9 \\
             \hline
             Theorem \ref{th:ICA}
              & $\mathbb F=\mathbb C$ & 4 & 9  &   16 & 25&  36& 49&  64\\
             \hline
             \cite[Eq. (15)]{2008LievenSOBIUM}
              & $\mathbb F=\mathbb C$ & 4 & 9 & 14 & 21 & 30& 40  & 51 \\
            \hline
            \cite[Proposition 1.11]{AlgGeom1} & $\mathbb F=\mathbb R$\rlap{\textsuperscript{*}} & 3 & 6 & 10 & 15 & 21& 28 & 36\\
            \hline 
		            \multicolumn{9}{l}{\scriptsize\textsuperscript{*}or $\mathbb F=\mathbb C$ if the superscript ``$H$'' in \eqref{eq:SOBI} is replaced by the superscript ``$T$''}                         
         \end{tabular}
     \end{table}

\section{An application in deterministic signal separation using mild source models} \label{sampled sources}
\subsection{Context and contribution}
We have recently proposed tensor-based algorithms for the deterministic blind separation of signals that can be modeled as exponential polynomials (i.e., sums and/or products of exponentials, sinusoids and/or polynomials) \cite{LL1Lieven2011} or as rational functions \cite{Otto_Lowner}. These signal models are meant to be little restrictive; on the other hand, they enable a unique source separation under certain conditions. The approach is somewhat related to sparse modelling \cite{BDE}. In sparse modelling, 
matrix $\M$ in \eqref{eq:first_eq} is known but has typically more columns than rows while most of the entries of $\Sm$ are zero. That is, the nonzero entries of $\Sm$ make sparse combinations of the columns of $\M$ (called the ``dictionary'') to model $\X$. The uniqueness of the model depends on the degree of linear independence of the columns of $\M$ and the degree of sparsity of the rows of $\Sm$  \cite{BDE}.
In \cite{LL1Lieven2011,Otto_Lowner} on the other hand, the basis vectors are estimated as well, by optimization over continuous variables. By way of example, in the case of sparse modelling of a sine wave, the columns of $\M$ could be chosen as sampled versions of 
$\sin((\omega_0 + k \Delta \omega) t)$ for a number of values $k$ (say $k=-K,\dots,-1,0,1,\dots,K$ so that $R=2K+1$), and $\omega_0$ and $\Delta \omega$ are fixed. On the other hand, in \cite{LL1Lieven2011} one optimizes over a continuous variable $\omega$ to determine the best representation $\sin(\omega t)$; in this way the accuracy is not bounded by $\Delta \omega$.

In \cite{LL1Lieven2011,Otto_Lowner} deterministic uniqueness conditions are given for exponential polynomial and rational source models. Here, we propose generic uniqueness conditions for the case that the mixing matrix has full column rank.

We actually consider a more general family of models, namely we assume that the source signals
$s_1(t),\dots,s_R(t)$ can be modeled as the composition of a known multivariate rational function and functions of the type $t$, $\cos (\omega t+\phi)$, $\sin (\omega t+\phi)$, and $a^t$.
We assume that the discrete-time signals are obtained by sampling at the points $t=1,\dots,N$.
The observed data are a mixture of the sources:
\begin{equation}\label{eq:sep_model}
\X=\M\left[\begin{matrix}s_1(1)&\dots & s_1(N)\\
\vdots& \vdots &\vdots\\
s_R(1)&\dots & s_R(N)
\end{matrix}\right]=\M\Sm^T.
\end{equation}
\subsection{An example}
To simplify the presentation we will consider the concrete case where the source signals can be modelled as
\begin{equation}\label{eq:example_s_r}
s_r(t)=\frac{a_r^t}{t}+\frac{b_r+t}{c_r+t}\cos(\alpha_r t+\phi_r)+\cos(\beta_r t),\quad t\in\mathbb R
\end{equation}
for a priori unknown parameters $a_r$, $b_r$, $c_r$, $\alpha_r$, $\phi_r$ and $\beta_r$.
That is, $s_r(t)$ is the composition of the known rational function
$$
R(x_1,\dots,x_6)=\frac{x_1}{x_2}+\frac{x_3+x_2}{x_4+x_2}x_5+x_6
$$
and the functions $x_1(t)=a_r^t$, $x_2(t)=t$, $x_3(t)=b_r$, $x_4(t)=c_r$, $x_5(t)=\cos (\alpha_r t+\phi_r)$, and $x_6=\cos \beta_rt$.
The general case can be studied similarly.

In the remaining part of this subsection we show that 
if (i) $R\leq N-6$, (ii) the parameters $a_r$, $b_r$, $c_r$, $\alpha_r$, $\phi_r$, and $\beta_r$ are generic, and (iii) the mixing matrix $\M$ has full column rank,  then the mixing matrix and the  sources $s_1(t),\dots, s_R(t)$ can be uniquely recovered.

We rewrite \eqref{eq:sep_model} as matrix decomposition \eqref{eq:sumstrucrank1_structured}. We set $\newX=\X$ and
$\newM(\z)=\M$.
It is clear that  the signals in \eqref{eq:example_s_r} can be parameterized as
 \begin{equation}\label{eq:example_s}
 s(t)=\frac{\zeta_1^t}{t}+\frac{\zeta_2+t}{\zeta_3+t}\cos(\zeta_4 t+\zeta_5)+\cos(\zeta_6 t),\quad t\in\mathbb R,
 \end{equation}
 where $\bmzeta= [\zeta_1\ \dots\ \zeta_6]^T=[a\ b\ c\ \alpha\ \phi\ \beta]^T$, so we set $\newsvec(\bmzeta)=[s(1)\ \dots\ s(N)]^T$.
First, we bring $\newsvec(\bmzeta)$ into the form \eqref{eq:column_s}. Then we will check
 assumptions  1--6  in Theorem \ref{th:main}. 
 
The following identities are well-known:
\begin{equation}\label{eq:cos_sine}
\cos\zeta=\frac{1-\tan^2\frac{\zeta}{2}}{1+\tan^2\frac{\zeta}{2}},\quad 
\sin\zeta=\frac{2\tan\frac{\zeta}{2}}{1+\tan^2\frac{\zeta}{2}}.
\end{equation} 
We will need the following generalization of \eqref{eq:cos_sine}.
\begin{lemma}\label{lemma:cos_sin}
There exist a polynomial $P_n$ and rational functions $Q_n$ and $R_n$ such that
\begin{align}
\cos \zeta n &= P_n(\cos\zeta) = Q_n\left(\tan\frac{\zeta}{2}\right),\label{eq:thelastequation1}\\
\sin \zeta n &= R_n\left(\tan\frac{\zeta}{2}\right).\label{eq:thelastequation2}
\end{align}
\end{lemma}
\begin{proof}
See Appendix B.
\end{proof}
From \eqref{eq:example_s} and Lemma \ref{lemma:cos_sin} it follows that
 \begin{equation*}
\begin{split}
&s(n) = \frac{\zeta_1^n}{n}+\frac{\zeta_2+n}{\zeta_3+n}\cos(\zeta_4 n+\zeta_5)+\cos(\zeta_6 n)=\\
&\frac{\zeta_1^n}{n}+\frac{\zeta_2+n}{\zeta_3+n}\left(\cos \zeta_4 n \cos\zeta_5-\sin \zeta_4 n \sin\zeta_5\right)+\cos(\zeta_6 n)=\\
&\frac{\zeta_1^n}{n}+\frac{\zeta_2+n}{\zeta_3+n}\left(Q_n\left(\tan\frac{\zeta_4}{2}\right) \frac{1-\tan^2\frac{\zeta_5}{2}}{1+\tan^2\frac{\zeta_5}{2}}-\right.\\
&\left.R_n\left(\tan\frac{\zeta_4}{2}\right) \frac{2\tan\frac{\zeta_5}{2}}{1+\tan^2\frac{\zeta_5}{2}}\right)+P_{n}(\cos\zeta_6)=
 \frac{p_n(\f(\bmzeta))}{q_n(\f(\bmzeta))},
\end{split}
\end{equation*}
where
\begin{align*}
&\frac{p_n(\bmzetax)}{q_n(\bmzetax)}=\frac{\zetax_1^n}{n}+\frac{\zetax_2+n}{\zetax_3+n}\left(Q_n(\zetax_4)\frac{1-\zetax_5^2}{1+\zetax_5^2}-\right.\\
&\left.\qquad\qquad\qquad\qquad\qquad\quad\qquad R_n(\zetax_4)\frac{2\zetax_5}{1+\zetax_5}\right)+P_n(\zetax_6),\\
& \f(\zeta_1,\zeta_2,\zeta_3,\zeta_4,\zeta_5,\zeta_6) = [\zeta_1\ \zeta_2\ \zeta_3\ \tan\frac{\zeta_4}{2}\ \tan\frac{\zeta_5}{2}\ \cos\zeta_6]^T. 
\end{align*}
Thus, 
$\newsvec(\bmzeta)=[s(1)\ \dots\ s(N)]^T$ is of the  form \eqref{eq:column_s} and $l=6$. Now we check
assumptions  1--6  in Theorem \ref{th:main}: 1) holds by our assumption (iii); 2) and 3) are trivial; 
4) holds for $\widehat N= N$ since the vectors $\newsvec(\zeta_1,\zeta_2,\dots,\zeta_6)-\newsvec(0,\zeta_2,\dots,\zeta_6)=[\frac{\zeta_1}{1}\ \dots\ \frac{\zeta_1^N}{N}]^T$ span the entire space $\mathbb F^N$;
 5) holds for $\widehat l=l=6$; 6) holds by assumption (i).

\subsection{Separation of exponential polynomials and separation of rational functions}
The cases where the sources in \eqref{eq:sep_model} can be expressed
as sampled exponential polynomials
\begin{equation}\label{eq:exp_poly}
\begin{split}
s(n) =  &\sum\limits_{f=1}^F (p_{0f} +p_{1f}n+\dots+p_{d_ff}n^{d_f})a_f^n=\\
&\sum\limits_{f=1}^F P_f(n)a_f^n,\quad n=1,\dots,N
\end{split}
\end{equation}
and sampled rational functions
\begin{equation}\label{eq:rat_funct}
s(n) = \frac{a_0+a_1n+\dots+a_pn^p}{b_0+b_1n+\dots+b_qn^q},\quad n=1,\dots,N
\end{equation}
were studied in \cite{LL1Lieven2011} and \cite{Otto_Lowner}, respectively.

The following two theorems
complement results on generic uniqueness from \cite{LL1Lieven2011} and \cite{Otto_Lowner}.
In contrast to papers \cite{LL1Lieven2011} and \cite{Otto_Lowner} we do not exploit specific properties
of Hankel or L\"owner matrices in our derivation. We only use the source models \eqref{eq:exp_poly}--\eqref{eq:rat_funct} for verifying 
the assumptions in Theorem \ref{th:main}.
\begin{theorem} \label{th:exp_poly}
Assume that the mixing matrix $\M$ has full column rank and  that
\begin{equation}\label{eq:bound_exp_poly}
R\leq N-(d_1+\dots d_F+2F),
\end{equation}
then $\M$ and $R$ generic sources of form \eqref{eq:exp_poly} can be uniquely recovered from 
the observed data $\X = \M\Sm^T$.
\end{theorem}
\begin{proof}
We set 
\begin{align*}
\bmzeta &= [a_1\  p_{01}\ \dots\ p_{d_11}\ \dots\ a_F\  p_{0F}\ \dots\ p_{d_FF}]^T\in\mathbb F^l,\\
 l &= (2+d_1)+\dots+(2+d_F)=d_1+\dots+d_F+2F
\end{align*}
and check the assumptions in Theorem \ref{th:main} for $\newX=\X$, $\newM(\z)=\M$ and $\newsvec(\bmzeta) = [s(1)\ \dots\ s(N)]^T$: 1)--3) are trivial; 4) since the vectors $\newsvec(\zeta,1,0,\dots,0)=[\zeta\ \dots\ \zeta^N]^T$ span the entire space $\mathbb F^N$, we set $\widehat{N}=N$; 5) we set
$\widehat{l}=l$; 6) holds by \eqref{eq:bound_exp_poly}
since $\operatorname{Range}(\br)$   is invariant under scaling.
\end{proof}
\begin{theorem} \label{th:rat_funct}
Assume that the mixing matrix $\M$ has full column rank, $q\geq 1$, and that
\begin{equation}\label{eq:bound_rat_funct}
R\leq N-(p+q+1),
\end{equation}
then $\M$ and  $R$ generic sources of form \eqref{eq:rat_funct} can be uniquely recovered from  the observed data $\X = \M\Sm^T$.
\end{theorem}
\begin{proof}
We set 
$$
\bmzeta = [a_0\ \dots\ a_p\ b_0\ \dots\ b_q]^T\in\mathbb F^l,\quad  l = p+q+2
$$
and check the  assumptions in Theorem \ref{th:main} for $\newX=\X$, $\newM(\z)=\M$ and $\br(\bmzeta)=\newsvec(\bmzeta) = [s(1)\ \dots\ s(N)]^T$: 1)--3) are trivial; 4) since an $N\times N$ matrix  with $(k+1)$th column (for $k=0,\dots,N-1$) given by
$$
\newsvec(\underbrace{1,0,\dots,0}_{p+1},k,1,0,\dots,0)=[(k+1)^{-1}\ \dots\ (k+N)^{-1}]^T,
$$
is nonsingular \cite[p. 38]{HornJohnson},  we set $\widehat{N}=N$; 5) an elementary computation shows that for a generic $\bmzetax$, $\J(\br,\bmzetax)\bmzetax=\bf 0$, implying that $\operatorname{rank}(\J(\br,\bmzetax))\leq l-1$, so we set  $\widehat{l}=l-1$;
6) holds by \eqref{eq:bound_rat_funct} since $\operatorname{Range}(\br)$   is invariant under scaling.
\end{proof}
We assume that the matrix $\M$
is generic and compare the bounds in Theorem \ref{th:exp_poly} and Theorem \ref{th:rat_funct} with the generic bounds in \cite{LL1Lieven2011} and 
\cite{Otto_Lowner}, respectively.
Since $\M$ is generic, it has full column rank
if and only if $R\leq K$. Thus, we compare the bound  $R\leq \min(N-(d_1+\dots d_F+2F),K)$  with the bound $R(d_1+\dots d_F+F)\leq\lfloor \frac{N+1}{2}\rfloor$, $2\leq K$ in \cite{LL1Lieven2011}, and the bound  $R\leq \min(N-(p+q+1),K)$ with the bound $R\leq\frac{1}{\max(p,q)}\lfloor \frac{N+1}{2}\rfloor$, $2\leq K$ in \cite{Otto_Lowner}. On one hand, the bounds in
\cite{LL1Lieven2011} and \cite{Otto_Lowner} can be used in the undetermined case  ($2\leq K$), while our bounds work only in the overdetermined case ($R\leq K$). On the other hand, roughly speaking, our bounds are of the form $R\leq N-c$ while the bounds in \cite{LL1Lieven2011} and 
\cite{Otto_Lowner} are of the form $R\leq N/c$,
where $c$ is the number of parameters that describe a generic signal. In this sense our new uniqueness conditions are significantly more relaxed.

\section{Conclusion}\label{sec:conclusion}
Borrowing insights from algebraic geometry, we have presented a theorem that can be used for investigating generic uniqueness in BSS problems that can be formulated as a particular structured matrix factorization. We have used this tool for deriving generic uniqueness conditions in (i) SOBIUM-type independent component analysis and (ii) a class of deterministic BSS approaches that rely on parametric source models. 
In a companion paper we will use the tool to obtain generic results for structured tensor and coupled matrix/tensor factorizations.

\appendices
\section{Proof of Theorem \ref{th:main}}\label{Appendix_A}
In this appendix we consider the decomposition
\begin{equation}\label{eq:general_decomposition}
\newX=\newM\newSm^T=\sum_{r=1}^R \newmvec_r\newsvec_r^T,\quad \newsvec_r\in S,
\end{equation}
where  the matrix $\newM$ has full column rank and  $S$ denotes a known subset of $\mathbb F^N$.

In Theorem \ref{th:main_deterministic}  below, we present two conditions that guarantee the uniqueness  of decomposition \eqref{eq:general_decomposition}.
These conditions will be checked 
in the proof of Theorem \ref{th:main} for generic points in $S=\{\newsvec(\bmzeta):\ q_1(\f(\bmzeta))\cdots q_N(\f(\bmzeta))\ne 0\}$, where $\newsvec(\bmzeta)$ is defined in  \eqref{eq:column_s}. The latter proof is given in Subsection \ref{subsec:appAC}. The step from the deterministic formulation in Subsection
\ref{subsec:appAA} to the generic result in Subsection \ref{subsec:appAC} is taken in Subsection \ref{subsec:appAB}.
\vspace{-1em}
\subsection{A deterministic uniqueness result}\label{subsec:appAA}
\begin{theorem}\label{th:main_deterministic} Assume that
\begin{enumerate}[\IEEEsetlabelwidth{Z}]
\item the matrix $\newM$ has full column rank;
\item the columns $\newsvec_1,\dots,\newsvec_R$ of the matrix $\newSm$ satisfy the following condition:
\begin{equation}\label{eq:U2condition}
\begin{split}
 &\text{if at least two of the values }\ \lambda_1,\dots,\lambda_R\in\mathbb F\\
&\text{are nonzero, then }\ \lambda_1\newsvec_1+\dots+\lambda_R\newsvec_R\not\in S.
\end{split}
\end{equation}
\end{enumerate}
Then  decomposition \eqref{eq:general_decomposition} is unique. 
\end{theorem}
\begin{proof}
We need to show that  if there exist $\overbar{\newM}$ and $\overbar{\newSm}$ such that
\begin{equation}\label{eq:altern_decomp}
\newX =\overbar{\newM}\overbar{\newSm}^T=\sum\limits_{r=1}^{R}\overbar{\newmvec}_r\overbar{\newsvec}_r^T,
\quad \overbar{\newsvec}_r\in S
\end{equation}
then decompositions \eqref{eq:general_decomposition} and \eqref{eq:altern_decomp} coincide up to  permutation of the rank-$1$ terms.

First we show that assumption 2 implies that $\newSm$ has full column rank. Assume that there exist $\lambda_1,\dots,\lambda_R$ for which
$\lambda_1\newsvec_1+\dots+\lambda_R\newsvec_R=\bzero$, such that at least one of these values being nonzero would imply that $\newSm$ does not have full column rank.

Then for any $\mu\not\in\{0, -\lambda_1\}$,
$
\frac{\lambda_1+\mu}{\mu}\newsvec_1+\frac{\lambda_2}{\mu}\newsvec_2+\dots+\frac{\lambda_R}{\mu}\newsvec_R=\newsvec_1\in S.
$
Hence, by assumption 2, at most one of the values $\lambda_1+\mu,\lambda_2,\dots,\lambda_R$ is nonzero.
Since $\mu\ne -\lambda_1$, we have that $\lambda_2=\dots\lambda_R=0$. Since $\lambda_1\newsvec_1=\lambda_1\newsvec_1+\dots+\lambda_R\newsvec_R=\bzero$, it follows that $\lambda_1=0$ or $\newsvec_1=\bzero$. One can easily verify that
$\newsvec_1=\bzero$ is in contradiction to assumption 2. Hence $\lambda_1=0$.
 Thus the matrix $\newSm$ has full column rank.

Since the matrices $\newM$ and $\newSm$ have full column rank, it follows from the identity
\begin{equation}\label{eq:X_MS_MS}
\newX =\newM\newSm^T=\overbar{\newM}\overbar{\newSm}^T
\end{equation}
that the matrices $\overbar{\newM}$ and $\overbar{\newSm}$ also have full column rank.
Hence, 
\begin{equation}\label{eq:MMS_eq_S}
\overbar{\newM}^\dagger
\newM
\newSm^T=
\overbar{\newSm}^T
\end{equation}
where
$
\overbar{\newM}^\dagger =
\left(
\overbar{\newM}^{H}\overbar{\newM}
\right)^{-1}\overbar{\newM}^{H}
$
denotes the left inverse of $\overbar{\newM}$. 
By  assumption 2, each row of the matrix $\overbar{\newM}^\dagger
\newM$ contains at most one nonzero entry.  Since the matrices
$\newSm$ and $\overbar{\newSm}$ have full column rank, the square matrix 
$\overbar{\newM}^\dagger \newM$ is nonsingular.
Thus, each row and each column of 
$\overbar{\newM}^\dagger \newM$ contains exactly one nonzero entry. Hence there exist an $R\times R$
nonsingular diagonal matrix $\bLambda$ and an $R\times R$ permutation matrix $\bP$ such that
$\overbar{\newM}^\dagger \newM=\bLambda\bP$.
From \eqref{eq:MMS_eq_S} it follows that
\begin{equation}\label{eq:LPS_S}
\bLambda\bP\newSm^T=
\overbar{\newM}^\dagger \newM
\newSm=\overbar{\newSm}^T.
\end{equation}
Substituting \eqref{eq:LPS_S} into \eqref{eq:X_MS_MS} and taking into account that
the matrix $\newSm$ has full column rank we obtain
\begin{equation}\label{eq:MPL_M}
\newM\bP^T\bLambda^{-1}=\overbar{\newM}.
\end{equation}
Equations \eqref{eq:LPS_S}--\eqref{eq:MPL_M} imply that
decompositions \eqref{eq:general_decomposition} and \eqref{eq:altern_decomp} coincide up to permutation of the rank-$1$ terms.
\end{proof}
Theorem \ref{th:main_deterministic}
has already been proved for the particular cases where decomposition \eqref{eq:general_decomposition} represents the CPD of a third-order tensor \cite[Section IV]{JiangSid2004},
the CPD of a partially symmetric of order higher than three \cite[Theorem 4.1]{StegemanSymm}, the CPD of an unstructured tensor of order higher than three \cite[Theorem 4.2]{Stegeman2010}, and the decomposition in multilinear rank-$(L,L,1)$ terms
\cite[Theorem 2.4]{LL1Lieven2011}.
\subsection{A generic variant of assumption 2 in Theorem \ref{th:main_deterministic}}\label{subsec:appAB}
Condition \eqref{eq:U2condition} means that 
 the subspace $\operatorname{span}\{\newsvec_1,\dots,\newsvec_R\}$ has dimension $R$ and may intersect the set $S$ only at ``trivial'' points $\lambda_r\newsvec_r$, that is
\begin{align}
\text{the vectors }\newsvec_1,\dots,\newsvec_R\ \text{are linearly independent and}\label{eq:U2condition11}\\
\operatorname{span}\{\newsvec_1,\dots,\newsvec_R\}\cap S\subseteq\{\lambda \newsvec_r:\ \lambda\in\mathbb F,\ 1\leq r\leq R\}.\label{eq:U2condition2}
\end{align}
Property \eqref{eq:U2condition2} is the key to proving uniqueness of \eqref{eq:general_decomposition}. We can easily find 
$\operatorname{span}\{\newsvec_1,\dots,\newsvec_R\}$ from the matrix $\newX$ if it can be assumed that the matrix $\newM$ has full column rank. On the other hand, property \eqref{eq:U2condition2} means that the only points in $\operatorname{span}\{\newsvec_1,\dots,\newsvec_R\}$ that have the hypothesized structure (encoded in the definition of the set $S$), are the vectors $\newsvec_r$, $1\leq r\leq R$ (up to trivial indeterminacies).
However, conditions \eqref{eq:U2condition} and \eqref{eq:U2condition2} are most often hard to check for particular points $\newsvec_1,\dots,\newsvec_R$. The checking may become easier if we focus on the generic case, and this is where algebraic geometry comes in. 
More precisely, if $S=V$ is an algebraic variety, then the classical trisecant lemma states that
if $R$ is sufficiently small, then \eqref{eq:U2condition2} holds for ``generic'' $\newsvec_1,\dots,\newsvec_R\in S$.
A set $V\subseteq \mathbb C^N$ is an algebraic variety if it is the set of solutions of a system of polynomial equations. It is clear that algebraic varieties form an interesting class of subsets of $\mathbb C^N$; however, is not easy to verify whether a given subset of $\mathbb C^N$ is a variety or not.
On the other hand, it is known that a set obtained by evaluating a known rational vector-function (such as $\br(\bmzetax)$ in \eqref{eq:column_r}) can be extended to a variety by taking the closure, i.e., by including its boundary. This is indeed what we will do in the proof of Lemma \ref{lemma:generic_con_2} below. 
First we give a formal statement of the trisecant lemma. 
\begin{lemma}\label{lemma:trisecant}(\cite[Corollary 4.6.15]{book_joins_and_intersections},\cite[Theorem 1.4]{Chiantini2002})
Let $V\subset \mathbb C^N$ be an irreducible algebraic variety and $R\leq \dim\operatorname{span}\{V\}-\dim V$ or $R\leq \dim\operatorname{span}\{V\}-\dim V-1$ depending on whether $V$ is invariant under scaling or not.
Let $\Z_V$ denote a set of points $(\bv_1,\dots,\bv_R)$ such that 
\begin{equation*} 
\operatorname{span}\{\bv_1,\dots,\bv_R\}\cap V\not\subset\{\lambda \bv_r:\ \lambda\in\mathbb C,\ 1\leq r\leq R\}.
\end{equation*}
Then the Zariski closure of $\Z_V$ is a 
proper subvariety of $V\times\dots\times V$ ($R$ times), that is,
there exists a polynomial $h(\bv_1,\dots,\bv_R)$ in $RN$ variables whose zero set does not contain 
$V\times\dots\times V$ but does contain $\Z_V$.
\end{lemma}
It is the last sentence in the trisecant lemma that makes it a powerful tool for proving generic properties. Let us explain in more detail how this works. We can use $\Z_V$ to denote a set that poses problems in terms of uniqueness, in the sense that 
$\operatorname{span}\{\bv_1,\dots,\bv_R\}$  does {\em not} intersect $V$ only in the points that correspond to the pure sources. The trisecant lemma  states now that $\Z_V$ belongs to the zero set of a polynomial $h$ that is not identically zero and hence nonzero almost everywhere, i.e. the problematic cases occur in a measure-zero situation.
In order to make the connection with Theorem \ref{th:main} we will need the following notations:
\begin{align*}
\operatorname{Range}(\newsvec):=&\{\newsvec(\bmzeta):\ q_1(\f(\bmzeta))\cdots q_N(\f(\bmzeta))\ne 0,\ \bmzeta\in\mathbb F^l\},\\
\operatorname{Range}(\br):=&\{\br(\bmzetax):\ q_1(\bmzetax)\cdots q_N(\bmzetax)\ne 0,\ \bmzetax\in\mathbb F^l\}.
\end{align*}
\begin{lemma}\label{lemma:generic_con_2}
Let assumptions  2--6 in Theorem \ref{th:main} hold. Then assumption 2 in Theorem \ref{th:main_deterministic}
holds for $S=\operatorname{Range}(\newsvec)$
and $\newsvec_1=\newsvec(\bmzeta_1),\dots,\newsvec_R=\newsvec(\bmzeta_R)\in S$, where the vectors  
$\bmzeta_1,\dots,\bmzeta_R\in\mathbb F^l$ are generic.
\end{lemma}
\begin{proof}
Since \eqref{eq:U2condition11}--\eqref{eq:U2condition2} is equivalent to \eqref{eq:U2condition} it is sufficient to
show that $\mu_{Rl}(W_{\newsvec})=\mu_{Rl}(\Z_{\newsvec})$=0, where 
\begin{gather*}
\begin{split}
W_{\newsvec}=\{[\bmzeta_1^T\ \dots\ \bmzeta_R^T]^T:
\newsvec_1=\newsvec(\bmzeta_1),\dots,\newsvec_R=\newsvec(\bmzeta_R)\\ 
\text{are linearly dependent}\},
\nonumber
\end{split}\\
\begin{split}
\Z_{\newsvec}=\{&[\bmzeta_1^T\ \dots\ \bmzeta_R^T]^T:\ \eqref{eq:U2condition2}\ \text {does not hold for }\\
&\ \ \newsvec_1=\newsvec(\bmzeta_1),\dots,\newsvec_R=\newsvec(\bmzeta_R)\}. 
\end{split}
\end{gather*}
It is a well-known fact that the zero set
of a nonzero analytic function on $\mathbb C^{Rl}$  has measure zero both on $\mathbb C^{Rl}$ and $\mathbb R^{Rl}$.
Thus, to prove $\mu_{Rl}(W_{\newsvec})=\mu_{Rl}(\Z_{\newsvec})$=0, we will show that there exist analytic functions $w$ and
$g$ of $Rl$ complex variables such that
\begin{align}
w\ \text{is not identically zero but vanishes on}\ W_{\newsvec},\label{eq:wfunc}\\
g\ \text{is not identically zero but vanishes on}\ \Z_{\newsvec}\label{eq:gfunc}.
\end{align}
We consider  the following three cases: 1) $\mathbb F=\mathbb C$ and $\f(\bmzeta)=\bmzeta$;
2) $\mathbb F=\mathbb C$ and $\f(\bmzeta)$ is arbitrary; 3)  $\mathbb F=\mathbb R$.

1)\ {\em Case $\mathbb F=\mathbb C$ and $\f(\bmzeta)=\bmzeta$}. In this case $\newsvec(\bmzeta)=\br(\bmzeta)$, thus, the sets 
$W_{\newsvec}$ and $\Z_{\newsvec}$ take the following form:
\begin{gather*}
\begin{split}
W_{\newsvec}=W_{\br}=\{[\bmzeta_1^T\ \dots\ \bmzeta_R^T]^T:
\newsvec_1=\br(\bmzeta_1),\dots,\newsvec_R=\br(\bmzeta_R)\\ 
\text{are linearly dependent}\},
\nonumber
\end{split}\\
\begin{split}
\Z_{\newsvec}=\Z_{\br}=\{[\bmzeta_1^T\ \dots\ \bmzeta_R^T]^T:\ \eqref{eq:U2condition2}\ \text {does not hold for }\\
 S=\operatorname{Range}(\br)\
\text{and } \ \newsvec_1=\br(\bmzeta_1),\dots,\newsvec_R=\br(\bmzeta_R)\}.
\end{split}
\end{gather*}
Here we prove that 
there exist polynomials $d_{num}$ and $h_{num}$ in $Rl$ variables such that
\eqref{eq:wfunc}--\eqref{eq:gfunc} hold for $w=d_{num}$ and $g=h_{num}$.

First we focus on $\Z_{\br}$. Let $V$ denote the Zariski closure of $\operatorname{Range}(\br)\subset\mathbb C^N$.
Since $\operatorname{Range}(\br)$ is the image of the open (hence irreducible) subset 
$\{\bmzeta:\ q_1(\bmzeta)\cdots q_N(\bmzeta)\ne 0,\ \bmzeta\in\mathbb C^l\}$ under the rational map
$$
\br:\ \bmzeta\mapsto \left[\frac{p_1(\bmzeta)}{q_1(\bmzeta)},\dots,\frac{p_N(\bmzeta)}{q_N(\bmzeta)}\right]^T,
$$
it follows that $\operatorname{Range}(\br)$ is also an irreducible set.
Hence $ V\subset\mathbb C^N$ is an irreducible variety and
 the dimension of $V$ is equal to  $\operatorname{rank}\J(\br,\bmzeta)$ at a generic point $\bmzeta\in\mathbb C^l$ \cite[p. 186]{lectures_on_curves}. Hence, by assumption 5 in Theorem \ref{th:main},
\begin{equation}\label{eq:dim_l}
\dim V\leq\widehat{l}.
\end{equation}
Since, by definition, $\operatorname{Range}(\br)$ consists of all vectors of the form \eqref{eq:column_r}, from assumption 4 in Theorem \ref{th:main} it follows that 
\begin{equation*}
\begin{split}
&\dim\operatorname{span}\operatorname{Range}(\br)=\\
&\dim\operatorname{span}\{\br(\bmzeta):\ q_1(\bmzeta)\cdots q_N(\bmzeta)\ne 0,\ \bmzeta\in\mathbb C^l\}\geq \widehat{N}.
\end{split}
\end{equation*}
Since $V\supseteq \operatorname{Range}(\br)$, it follows that
\begin{equation}\label{eq:dim_span}
\dim\operatorname{span} V \geq \dim\operatorname{span}\operatorname{Range}(\br)\geq\widehat{N}.
\end{equation}
Thus, by assumption 6 in Theorem \ref{th:main} and \eqref{eq:dim_l}--\eqref{eq:dim_span},
$$
R\leq\widehat{N}-\widehat{l}\leq \dim\operatorname{span} V-\dim V.
$$
Thus, we have shown that $V$ satisfies the assumptions in Lemma \ref{lemma:trisecant}. Let now the set 
$\Z_V$ and the polynomial $h(\bv_1,\dots,\bv_R)$ be as in Lemma \ref{lemma:trisecant}.
Since $V$ is the Zariski closure of $\operatorname{Range}(\br)$, it follows that
 $V\times\dots\times V$ is the Zariski closure of $\operatorname{Range}(\br)\times\dots\times \operatorname{Range}(\br)$.
 Since, by Lemma \ref{lemma:trisecant}, the zero set of $h(\bv_1,\dots,\bv_R)$ does not contain $V\times\dots\times V$,
it follows that the zero set of $h(\bv_1,\dots,\bv_R)$ does not contain $\operatorname{Range}(\br)\times\dots\times \operatorname{Range}(\br)$.
 Hence, there exist
 $\bmzeta_1^0,\dots,\bmzeta_R^0\in\mathbb C^l$ such that  
$
 h(\br(\bmzeta_1^0),\dots,\br(\bmzeta_R^0))\ne 0.
 $
On the other hand, since $\operatorname{Range}(\br)$ is a subset of $V$, from the definitions of $\Z_{\br}$ and $\Z_V$ it follows that
$$
(\br(\bmzeta_1),\dots,\br(\bmzeta_R))\in \Z_V \ \text{ for all }\ [\bmzeta_1^T\ \dots\ \bmzeta_R^T]^T\in \Z_{\br}.
$$
Hence, by Lemma  \ref{lemma:trisecant}, 
\begin{equation}\label{eq:h_not_zero}
h(\br(\bmzeta_1),\dots,\br(\bmzeta_R)) =0\ \text{ for all }\ [\bmzeta_1^T\ \dots\ \bmzeta_R^T]^T\in \Z_{\br}.
\end{equation}
Since the function $h(\br(\bmzeta_1),\dots,\br(\bmzeta_R))$ is a composition of the polynomial $h$ in $RN$ variables  and $RN$ rational functions
$\frac{p_1(\bmzeta_1)}{q_1(\bmzeta_1)},\dots,\frac{p_N(\bmzeta_R)}{q_N(\bmzeta_R)}
$,
 it follows that $h(\br(\bmzeta_1),\dots,\br(\bmzeta_R))$ can be written as a ratio of two polynomials in the entries of $\bmzeta_1,\dots,\bmzeta_R$,
\begin{equation}\label{eq:h_ratio}
h(\br(\bmzeta_1),\dots,\br(\bmzeta_R))=\frac{h_{num}(\bmzeta_1,\dots,\bmzeta_R)}{h_{den}(\bmzeta_1,\dots,\bmzeta_R)}.
\end{equation}
By \eqref{eq:h_not_zero}--\eqref{eq:h_ratio}, $h_{num}$  vanishes on $\Z_{\br}$ and is not identically zero.
That is, \eqref{eq:gfunc} holds for $g=h_{num}$.

Now we focus on $W_{\br}$. By assumption 6 in Theorem \ref{th:main}, $\widehat{N}\geq R+\widehat{l}$, so assumption 4 implies that there exist $\bmzeta_1^0,\dots,\bmzeta_R^0\in\mathbb C^l$ such that the vectors $\br(\bmzeta_1^0),\dots,
\br(\bmzeta_R^0)$ are linearly independent. Hence there exists an $R\times R$ submatrix $\Rsub(\bmzeta_1,\dots,\bmzeta_R)$ of $[\br(\bmzeta_1)\ \dots\ \br(\bmzeta_R)]$ whose determinant $d(\bmzeta_1,\dots,\bmzeta_R)$ is not zero at the point  $(\bmzeta_1^0,\dots,\bmzeta_R^0)$.
On the other hand, $d(\bmzeta_1,\dots,\bmzeta_R)$ vanishes on $W_{\br}$ by definition. Since a determinant  is a multivariate polynomial, and since the entries of $\Rsub(\bmzeta_1,\dots,\bmzeta_R)$ are rational functions of $\bmzeta_1,\dots,\bmzeta_R$,
$d(\bmzeta_1,\dots,\bmzeta_R)$ can be written as a ratio of two polynomials
$d_{num}$ and $d_{den}$  in the entries of $\bmzeta_1,\dots,\bmzeta_R$.
It is clear that $d_{num}(\bmzeta_1^0,\dots,\bmzeta_R^0)\ne 0$ and that $d_{num}$ vanishes on $W_{\br}$.
That is, \eqref{eq:wfunc} holds for $w=d_{num}$.

2)\ {\em Case $\mathbb F=\mathbb C$ and $\f(\bmzeta)$ is arbitrary.} 
 We restrict ourselves to the case $\Z_{\newsvec}$. Namely, we use the polynomial $h_{num}$ and the function $\f$ to construct an analytic function $g=u_{num}$ in $Rl$ variables that satisfies  \eqref{eq:gfunc}. 
 The function $w$ that satisfies \eqref{eq:wfunc} can be constructed in the same way as $g$ but
 from the polynomial $d_{num}$ and the function $\f$.
 
 First we prove the existence and analyticity of  $g$.
 From the definitions of $\Z_{\newsvec}$ and $\Z_{\br}$ it follows that if $(\bmzeta_1,\dots,\bmzeta_R)\in \Z_{\newsvec}$, then
 $(\f(\bmzeta_1),\dots,\f(\bmzeta_R))\in \Z_{\br}$.
 Hence, by case 1 above and assumption 2 in Theorem \ref{th:main}, the set $\Z_{\newsvec}$ is contained in the zero set of the function 
 \begin{equation}\label{eq:g_eq_h}
 \begin{split}
 \g(\bmzeta_1,\dots,\bmzeta_R)&=
 h_{num}\left(\f(\bmzeta_1),\dots,\f(\bmzeta_R)\right)=
 \\
 & h_{num}\left(
 \frac{ f_{1,num}(\bmzeta_1)}{ f_{1,den}(\bmzeta_1)},\dots,\frac{f_{l,num}(\bmzeta_R)}{f_{l,den}(\bmzeta_R)}
  \right).
  \end{split}
  \end{equation}
Since $h_{num}$ is a polynomial, the function $\g$ can be represented as a ratio $\g=\g_{num}/\g_{den}$, where the functions $\g_{num}$ and $\g_{den}$ are defined on the whole space $\mathbb C^{Rl}$. Since both $\g_{num}$ and $\g_{den}$ consist of the composition of some polynomials and $2Rl$    functions $f_{1,num}(\bmzeta_1),f_{1,den}(\bmzeta_1),\dots,f_{l,num}(\bmzeta_R),f_{l,den}(\bmzeta_R)$ which are analytic on $\mathbb C^{RL}$, it follows  that $\g_{num}$ and $\g_{den}$ are analytic on $\mathbb C^{RL}$ \cite[p. 6]{GunningandRossi1965}.
We set $g=\g_{num}$. It is clear that $g$ vanishes on $\Z_{\newsvec}$.

Now we prove that $g$ is not identically zero. Since $\g=g/\g_{den}$, it is sufficient to  show that $\g$ is not zero at some point.
Let $\bmzeta^0$ be a point as in assumption 3 in Theorem \ref{th:main}. Then, by the inverse function theorem,
there exists a neighborhood $\mathcal N(\bmzeta^0,\varepsilon)\subset\mathbb F^l$ of the point $\bmzeta^0$ such that for any $\pointba\in \mathcal N(\bmzeta^0,\varepsilon)$
the equation $\f(\bmzeta)=\pointba$ has the solution $\bmzeta=\f^{-1}(\pointba)$.
Hence the equation $(\f(\bmzeta_1),\dots,\f(\bmzeta_R))=(\pointba_1,\dots,\pointba_R)$ has the solution
$(\bmzeta_1,\dots,\bmzeta_R)=(\f^{-1}(\pointba_1),\dots,\f^{-1}(\pointba_R))$ for all
$(\pointba_1,\dots,\pointba_R)\in \mathcal N(\bmzeta^0,\varepsilon)\times\dots\times \mathcal N(\bmzeta^0,\varepsilon)$.
Since $\mu_{Rl}(\mathcal N(\bmzeta^0,\varepsilon)\times\dots\times \mathcal N(\bmzeta^0,\varepsilon))=\mu_l(\mathcal N(\bmzeta^0,\varepsilon))^R>0$ \cite[Theorem B, p.144]{Halmos1950} and, by step 1), $h_{num}$ is not identically zero, there exists  a point $(\pointba_1^0,\dots,\pointba_R^0)\in \mathcal N(\bmzeta^0,\varepsilon)\times\dots\times \mathcal N(\bmzeta^0,\varepsilon)$ such that $h_{num}(\pointba_1^0,\dots,\pointba_R^0)\ne 0$. Hence, by 
\eqref{eq:g_eq_h},
$
\g(\f^{-1}(\pointba_1^0),\dots,\f^{-1}(\pointba_R^0))=
h_{num}(\pointba_1^0,\dots,\pointba_R^0)\ne 0.
$
That is, \eqref{eq:gfunc} holds for $g=\g_{num}$.

3)\ {\em Case $\mathbb F=\mathbb R$.} To distinguish between the complex and the real case  we  denote $\Z_{\newsvec}$ and $W_{\newsvec}$ in case 3 by $\Z_{\newsvec,\mathbb R}$ and $W_{\newsvec,\mathbb R}$, respectively. Similarly, the sets $\Z_{\newsvec}$ and $W_{\newsvec}$ considered in case 2, i.e. for $\mathbb F=\mathbb C$, are denoted by $\Z_{\newsvec,\mathbb C}$ and $W_{\newsvec,\mathbb C}$, respectively. 
Let $g_{\mathbb C}=\g_{num}$ and $w_{\mathbb C}=w$ denote the analytic functions constructed in case 2. Then
$g_{\mathbb C}$ and $w_{\mathbb C}$ are not identically zero and
$g_{\mathbb C}$ vanishes on $\Z_{\newsvec,\mathbb C}$ and $w_{\mathbb C}$ vanishes on $W_{\newsvec,\mathbb C}$.

Since $W_{\newsvec,\mathbb R}$ is a subset of $W_{\newsvec,\mathbb C}$, it follows that 
$w_{\mathbb C}$ vanishes on $W_{\newsvec,\mathbb R}$. Thus, \eqref{eq:wfunc} holds for  $w=
w_{\mathbb C}$.

It has not been proven that set $\Z_{\newsvec,\mathbb R}$ is a subset of $\Z_{\newsvec,\mathbb C}$ but in any case
$\Z_{\newsvec,\mathbb R}=\left(\Z_{\newsvec,\mathbb R}\cap \Z_{\newsvec,\mathbb C}\right)\cup \left(\Z_{\newsvec,\mathbb R}\setminus \Z_{\newsvec,\mathbb C}\right)$. We show that \eqref{eq:gfunc} holds for $g=g_{\mathbb C}\cdot w_{\mathbb C}$. Indeed, by case 2, $g_{\mathbb C}$ and hence $g=
g_{\mathbb C}w_{\mathbb C}$ vanish on $\Z_{\newsvec,\mathbb C}\supseteq\Z_{\newsvec,\mathbb R}\cap \Z_{\newsvec,\mathbb C}$. 
On the other hand, if $[\bmzeta_1^T\ \dots\ \bmzeta_R^T]^T\in \Z_{\newsvec,\mathbb R}\setminus \Z_{\newsvec,\mathbb C}$, then
there exist $\lambda_1,\dots,\lambda_R\in\mathbb R$, $\lambda\in\mathbb C\setminus\mathbb R$ and $r\in\{1,\dots,R\}$ such that
$\lambda_1\newsvec(\bmzeta_1)+\dots+\lambda_R\newsvec(\bmzeta_R)=\lambda\newsvec(\bmzeta_r)$, yielding that 
$[\bmzeta_1^T\ \dots\ \bmzeta_R^T]^T\in W_{\newsvec,\mathbb C}$. Thus, 
$\Z_{\newsvec,\mathbb R}\setminus \Z_{\newsvec,\mathbb C}\subseteq W_{\newsvec,\mathbb C}$ and $w_{\mathbb C}$ vanishes on $\Z_{\newsvec,\mathbb R}\setminus \Z_{\newsvec,\mathbb C}$ as well. That is, \eqref{eq:gfunc} holds for $g=g_{\mathbb C}w_{\mathbb C}$.
\end{proof}
\subsection{Proof of Theorem \ref{th:main}}\label{subsec:appAC}
We show that for a generic $\z\in\Omega$ that satisfies the conditions in Theorem \ref{th:main}, 
conditions 1--2 in Theorem \ref{th:main_deterministic} are also satisfied for 
\begin{align*}
&\newM=\newM(\z),\quad S=\{\newsvec(\bmzeta):\ q_1(\f(\bmzeta))\cdots q_N(\f(\bmzeta))\ne 0\}, \text{ and}\\
 &\newsvec_1 =\newsvec (\bmzeta_1),\dots, \newsvec_R=\newsvec (\bmzeta_R),
\end{align*}
where, by our notational convention from Subsection \ref{subsectionIC}, the vectors $\bmzeta_1,\dots,\bmzeta_R$ are such that $[\bmzeta_1^T\ \dots\ \bmzeta_R^T]^T$ coincides with the last  $s=Rl$
entries of $\z\in\Omega$.
The generic uniqueness that we want to prove in Theorem \ref{th:main},
then follows from Theorem  \ref{th:main_deterministic}.
We have the following.

1)  Condition 1 of Theorem \ref{th:main_deterministic} holds for generic $\z\in\Omega$  by  assumption 1 in Theorem \ref{th:main}.

2) By Lemma \ref{lemma:generic_con_2}, condition 2 of Theorem \ref{th:main_deterministic} holds for generic $\bmzeta_1,\dots,\bmzeta_R\in\mathbb F^l$, or equivalently, for generic $[\bmzeta_1^T\ \dots\ \bmzeta_R^T]^T\in\mathbb F^s$.
Hence, condition 2 of Theorem \ref{th:main_deterministic} holds for generic $\z\in\Omega$.
(Indeed, if $\widetilde{\Omega}$ denotes a set of points $\z\in\Omega$
such that condition 2 of Theorem \ref{th:main_deterministic} does not hold
and $\pi_{s}$ denotes the projection onto the last $s$
coordinates of $\mathbb F^n$, then, by Lemma \ref{lemma:generic_con_2}, $\mu_s\{\pi_s(\widetilde{\Omega})\}=0$, which implies \cite[Theorem B, p.144]{Halmos1950} that $\mu_n\{\widetilde{\Omega}\}=0$.) 

\section{Proof of Lemma \ref{lemma:cos_sin}}\label{Appendix_B}
We  use the fact that $\cos n\zeta$ and $\frac{\sin n\zeta}{\sin\zeta}$ are polynomials in
$\cos\zeta$ \cite[p. 642]{Prudnikov_vol_I}:
\begin{align*}
\cos n\zeta =& \sum\limits_{k=0}^{\lfloor n/2\rfloor} C^{2k}_n (\cos^2 \zeta-1)^k\cos^{n-2k}\zeta=P_{n}(\cos\zeta),\\
\frac{\sin n\zeta}{\sin\zeta} =& \sum\limits_{k=0}^{\lfloor(n-1)/2\rfloor} C^{2k+1}_n (\cos^2 \zeta-1)^k\cos^{n-2k-1}\zeta,
\end{align*}
where $\lfloor x\rfloor$ denotes  the 
integer part of $x$.
Substituting \eqref{eq:cos_sine}
into these equations we obtain that there exist rational functions $R_n$ and $Q_n$ such that
\eqref{eq:thelastequation1}--\eqref{eq:thelastequation2} hold.

\section*{Acknowledgment}
The authors wish to thank Giorgio Ottaviani and  Ed Dewey for their
assistance in algebraic geometry.


\bibliographystyle{IEEEtran}

%

%

\begin{IEEEbiography}[{\includegraphics[width=1in,height=1.25in,clip,keepaspectratio]{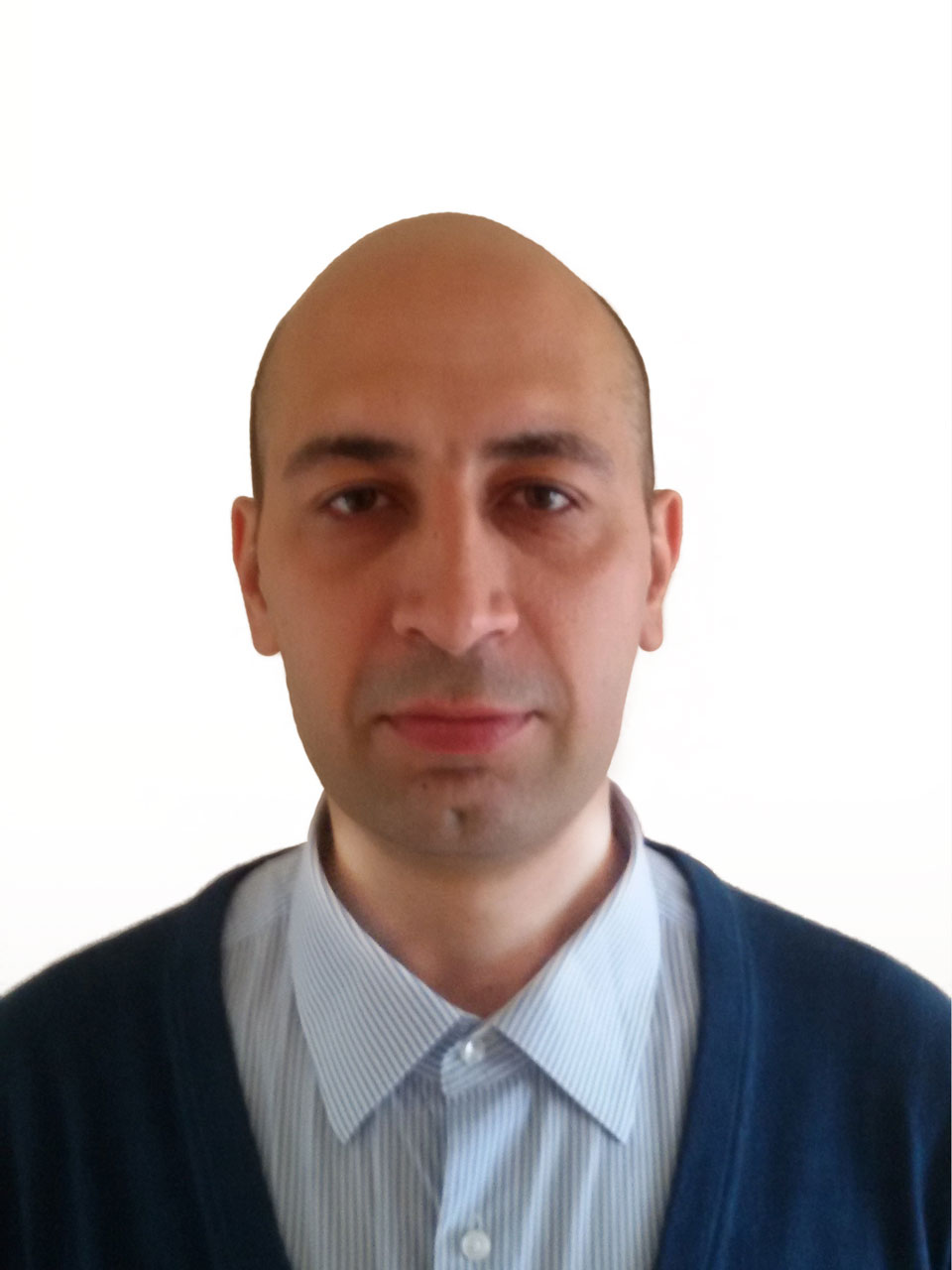}}]{Ignat Domanov}
Ignat Domanov received the Master’s degree
from Donetsk State University, Ukraine, the Ph.D degree in Physics and Mathematics
from Institute of Applied Mathematics and Mechanics, Ukraine, and
the Ph.D. degree in Engineering from KU Leuven, Belgium, 
in 1998, 2004, and 2013, respectively. Since 2013 he has been a Postdoctoral
Fellow with the KU Leuven, Belgium. His
research interests include applied linear algebra,
tensor decompositions and tensor-based
signal processing.
\end{IEEEbiography}
\begin{IEEEbiography}[{\includegraphics[width=1in,height=1.25in,clip,keepaspectratio]{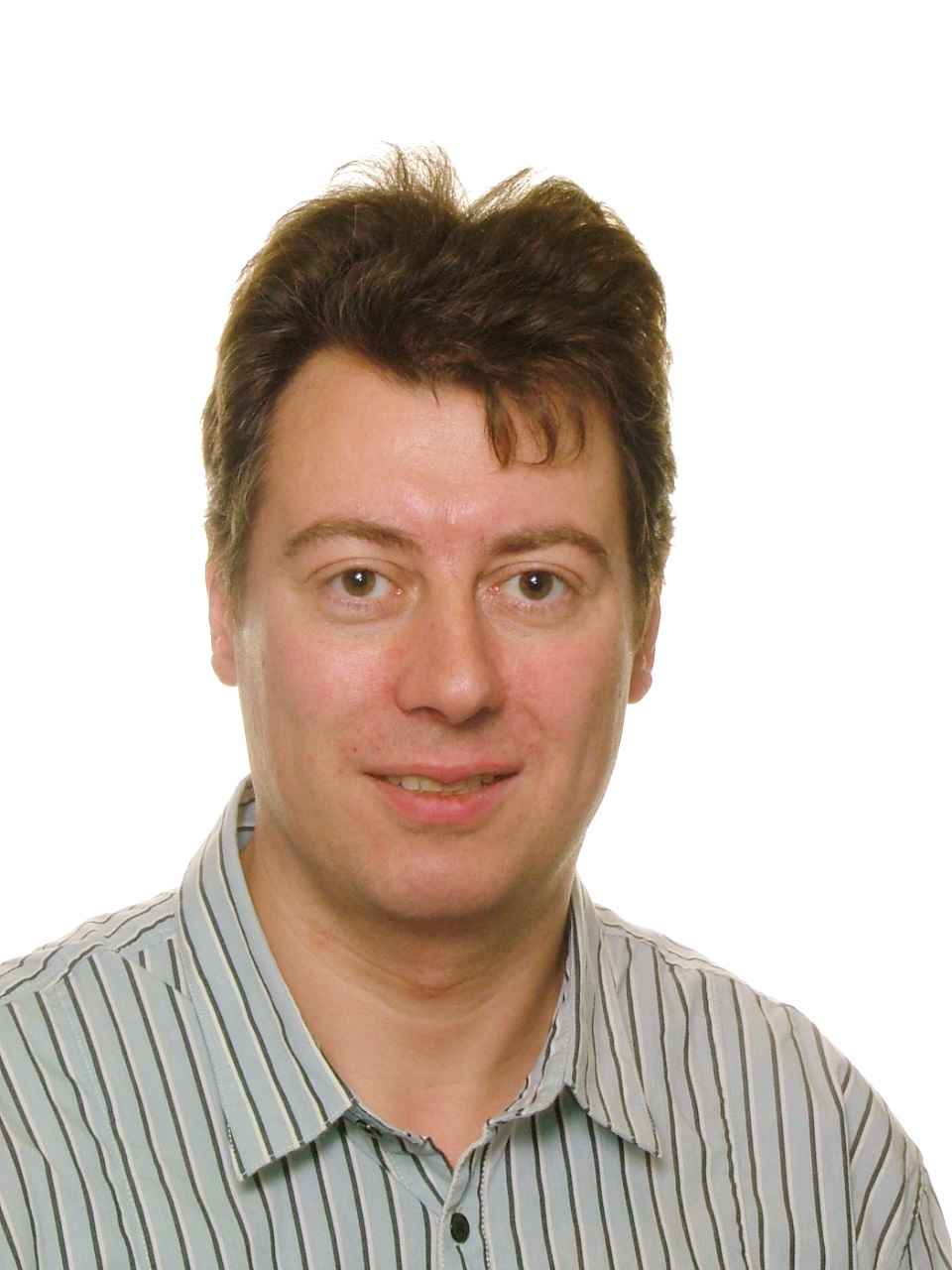}}]
{Lieven De Lathauwer}
Lieven De Lathauwer received the Master’s degree in electromechanical
engineering and the Ph.D. degree in applied sciences from
KU Leuven, Belgium, in 1992 and 1997, respectively. He is currently
Professor with KU Leuven, Belgium. Dr. De Lathauwer is an Associate
Editor of the SIAM Journal on Matrix Analysis and Applications
and has been an Associate Editor for IEEE Transactions on Signal
Processing. His research concerns the development of tensor tools for
engineering applications.
\end{IEEEbiography}






\end{document}